\documentclass[12pt]{amsart}

\title[Algebraicity and implicit definability]{Algebraicity and implicit definability\\ in set theory}

\author{Joel David Hamkins}
  \address{J. D. Hamkins, Mathematics, Philosophy, Computer Science, The Graduate Center of The City University of New York, 365 Fifth Avenue, New York, NY 10016
           \& Mathematics, The College of Staten Island of CUNY, Staten Island, NY 10314}
  \email{jhamkins@gc.cuny.edu, http://jdh.hamkins.org}

\author{Cole Leahy}
	\address{C. Leahy, Department of Linguistics \& Philosophy, 77 Massachusetts Avenue, Cambridge, MA 02139}
	\email{cleahy@mit.edu, http://cel.mit.edu}

\thanks{This project began when the first author responded to a question posed by the second author on the website MathOverflow \cite{MO71537Leahy:PointwiseAlgebraicModelsOfSetTheory}. The first author's research has been supported in part by
National Science Foundation program grant DMS-0800762, by Simons Foundation grant 209252, by PSC-CUNY grant 66563-00 44 and by grant 80209-06 20 from the CUNY Collaborative Incentive Award program. The authors thank Leo Harrington for an insightful conversation with the first author during math tea at the National University of Singapore in July of 2011. Commentary on this paper can be made at http://jdh.hamkins.org/algebraicity-and-implicit-definability.}

\usepackage{latexsym,amsfonts,amsmath,amssymb}
\usepackage[hidelinks]{hyperref}
\newtheorem{theorem}{Theorem}

\newtheorem{corollary}[theorem]{Corollary}

\newtheorem{observation}[theorem]{Observation}

\newcommand{\QED}{\end{proof}}

\def\proclaim[#1]{{\bf #1}}
\def\BF#1.{{\bf #1.}}

\newcommand{\Godel}{G\"odel}

\newcommand{\Levy}{L\'{e}vy}

\newcommand{\of}{\subseteq}

\newcommand{\set}[1]{\{\,{#1}\,\}}

\newcommand{\Coll}{\mathop{\rm Coll}}

\newcommand{\Th}{\mathop{\rm Th}}

\newcommand{\satisfies}{\models}

\newcommand{\gHOD}{\ensuremath{\mathord{{\rm g}\HOD}}} 

\newcommand{\Union}{\bigcup}

\newcommand{\intersect}{\cap}

\newcommand{\smalllt}{\mathrel{\mathchoice{\raise2pt\hbox{$\scriptstyle<$}}{\raise1pt\hbox{$\scriptstyle<$}}{\raise0pt\hbox{$\scriptscriptstyle<$}}{\scriptscriptstyle<}}}
\newcommand{\smallleq}{\mathrel{\mathchoice{\raise2pt\hbox{$\scriptstyle\leq$}}{\raise1pt\hbox{$\scriptstyle\leq$}}{\raise1pt\hbox{$\scriptscriptstyle\leq$}}{\scriptscriptstyle\leq}}}
\newcommand{\lt}{\smalllt}

\newcommand{\boolval}[1]{\mathopen{\lbrack\!\lbrack}\,#1\,\mathclose{\rbrack\!\rbrack}}
\def\[#1]{\boolval{#1}}
\newbox\gnBoxA
\newdimen\gnCornerHgt
\setbox\gnBoxA=\hbox{\tiny$\ulcorner$}
\global\gnCornerHgt=\ht\gnBoxA
\newdimen\gnArgHgt
\def\gcode #1{%
\setbox\gnBoxA=\hbox{$#1$}%
\gnArgHgt=\ht\gnBoxA%
\ifnum     \gnArgHgt<\gnCornerHgt \gnArgHgt=0pt%
\else \advance \gnArgHgt by -\gnCornerHgt%
\fi \raise\gnArgHgt\hbox{\tiny$\ulcorner$} \box\gnBoxA %
\raise\gnArgHgt\hbox{\tiny$\urcorner$}}
\newcommand{\UnderTilde}[1]{{\setbox1=\hbox{$#1$}\baselineskip=0pt\vtop{\hbox{$#1$}\hbox to\wd1{\hfil$\sim$\hfil}}}{}}
\newcommand{\Undertilde}[1]{{\setbox1=\hbox{$#1$}\baselineskip=0pt\vtop{\hbox{$#1$}\hbox to\wd1{\hfil$\scriptstyle\sim$\hfil}}}{}}
\newcommand{\undertilde}[1]{{\setbox1=\hbox{$#1$}\baselineskip=0pt\vtop{\hbox{$#1$}\hbox to\wd1{\hfil$\scriptscriptstyle\sim$\hfil}}}{}}
\newcommand{\UnderdTilde}[1]{{\setbox1=\hbox{$#1$}\baselineskip=0pt\vtop{\hbox{$#1$}\hbox to\wd1{\hfil$\approx$\hfil}}}{}}
\newcommand{\Underdtilde}[1]{{\setbox1=\hbox{$#1$}\baselineskip=0pt\vtop{\hbox{$#1$}\hbox to\wd1{\hfil\scriptsize$\approx$\hfil}}}{}}

\newcommand{\st}{\mid}
\renewcommand{\th}{{\hbox{\scriptsize th}}}

\renewcommand{\iff}{\mathrel{\leftrightarrow}}

\def\<#1>{\langle#1\rangle}

\newcommand{\Ord}{\mathop{{\rm Ord}}}

\newcommand{\ZFC}{{\rm ZFC}}
\newcommand{\ZF}{{\rm ZF}}

\newcommand{\KM}{{\rm KM}}

\newcommand{\GBC}{{\rm GBC}}

\newcommand{\CH}{{\rm CH}}

\newcommand{\GCH}{{\rm GCH}}

\newcommand{\HOD}{{\rm HOD}}
\newcommand{\HOA}{{\rm HOA}}

\newcommand{\cell}[1]{\boxit{\hbox to 17pt{\strut\hfil$#1$\hfil}}}
\newcommand{\head}[2]{\lower2pt\vbox{\hbox{\strut\footnotesize\it\hskip3pt#2}\boxit{\cell#1}}}
\newcommand{\boxit}[1]{\setbox4=\hbox{\kern2pt#1\kern2pt}\hbox{\vrule\vbox{\hrule\kern2pt\box4\kern2pt\hrule}\vrule}}
\newcommand{\Col}[3]{\hbox{\vbox{\baselineskip=0pt\parskip=0pt\cell#1\cell#2\cell#3}}}
\newcommand{\tapenames}{\raise 5pt\vbox to .7in{\hbox to .8in{\it\hfill input: \strut}\vfill\hbox to
.8in{\it\hfill scratch: \strut}\vfill\hbox to .8in{\it\hfill output: \strut}}}
\newcommand{\Head}[4]{\lower2pt\vbox{\hbox to25pt{\strut\footnotesize\it\hfill#4\hfill}\boxit{\Col#1#2#3}}}
\newcommand{\Dots}{\raise 5pt\vbox to .7in{\hbox{\ $\cdots$\strut}\vfill\hbox{\ $\cdots$\strut}\vfill\hbox{\
$\cdots$\strut}}}
\newcommand{\df}{\it} 
\newcommand{\OD}{{\rm OD}}
\newcommand{\OA}{{\rm OA}}
\newcommand{\Imp}{{\rm Imp}}
\newcommand{\Palg}{P_{\rm alg}}
\newcommand{\Pimp}{P_{\rm imp}}
\newcommand{\Pdef}{P_{\rm def}}
\newcommand{\gImp}{{\rm gImp}}
\renewcommand{\satisfies}{\vDash}

\newbox\gnBoxA
\newdimen\gnCornerHgt
\setbox\gnBoxA=\hbox{$\ulcorner$}
\global\gnCornerHgt=\ht\gnBoxA
\newdimen\gnArgHgt
\def\Godelnum #1{%
\setbox\gnBoxA=\hbox{$#1$}%
\gnArgHgt=\ht\gnBoxA%
\ifnum     \gnArgHgt<\gnCornerHgt \gnArgHgt=0pt%
\else \advance \gnArgHgt by -\gnCornerHgt%
\fi \raise\gnArgHgt\hbox{$\ulcorner$} \box\gnBoxA %
\raise\gnArgHgt\hbox{$\urcorner$}}

\begin{document}


\begin{abstract}
We analyze the effect of replacing several natural uses of definability in set theory by the weaker model-theoretic notion of algebraicity. We find, for example, that the class of hereditarily ordinal algebraic sets is the same as the class of hereditarily ordinal definable sets, that is, $\HOA = \HOD$. Moreover, we show that every (pointwise) algebraic model of $\ZF$ is actually pointwise definable. Finally, we consider the implicitly constructible universe $\Imp$---an algebraic analogue of the constructible universe---which is obtained by iteratively adding not only the sets that are definable over what has been built so far, but also those that are algebraic (or equivalently, implicitly definable) over the existing structure. While we know $\Imp$ can differ from $L$, the subtler properties of this new inner model are just now coming to light. Many questions remain open.
\end{abstract}

\maketitle

\thispagestyle{empty}

\noindent
We aim here to analyze the effect of replacing several natural uses of definability in set theory by the weaker model-theoretic notion of algebraicity and its companion concept of implicit definability. In place of the class $\HOD$ of hereditarily ordinal definable sets, for example, we consider the class $\HOA$ of hereditarily ordinal-algebraic sets. In place of the pointwise definable models of set theory, we examine its (pointwise) algebraic models. And in place of \Godel's constructible universe $L$, obtained by iterating the definable power set operation, we introduce the implicitly constructible universe \Imp, obtained by iterating the algebraic or implicitly definable power set operation. In each case we investigate how the change from definability to algebraicity affects the nature of the resulting concept. We are especially intrigued by \Imp, a new inner model of \ZF\ whose subtler properties are just now coming to light. Open questions about \Imp\ abound.

Before proceeding further, let us review the basic definability definitions. In the model theory of first-order logic, an element $a$ is {\df definable} (without parameters) in a structure $M$ if it is the unique object in $M$ satisfying some first-order property $\varphi$ there, that is, if $M\satisfies\varphi[b]$ just in case $b=a$. More generally, an element $a$ is {\df algebraic} in $M$ if it has a property $\varphi$ exhibited by only finitely many objects in $M$, so that $\set{b\in M \st M\satisfies\varphi[b]}$ is a finite set containing $a$.
For each class $P\of M$ we can similarly define what it means for an element to be $P$-definable or $P$-algebraic by allowing the formula $\varphi$ to have parameters from $P$. Since each element of a structure $M$ for a language with equality is trivially $M$-definable, the notion of a $P$-definable element is interesting only when the inclusion $P\of M$ is proper.

In the second-order context, a subset or class $A\of M^n$ is said to be {\df definable} in $M$ (again without parameters, unless otherwise specified), if $A=\set{\vec a\in M\st M\satisfies\varphi[\vec a]}$ for some first-order formula $\varphi$. In particular, $A$ is the unique class in $M^n$ with $\<M,A>\satisfies\forall \vec x\, [\varphi(\vec x)\iff A(\vec x)]$, in the language where we have added a predicate symbol for $A$. Generalizing this condition, we say that a class $A\of M^n$ is {\df implicitly definable} in $M$ if there is a first-order formula $\psi(A)$ in the expanded language, not necessarily of the form $\forall \vec x\, [\varphi(\vec x)\iff A(\vec x)]$, with the property that $A$ is the unique class for which $\<M,A>\satisfies\psi(A)$. While every (explicitly) definable class is thus also implicitly definable, we will see below that the converse can fail. Generalizing even more, we say that a class $A\of M^n$ is {\df algebraic} in $M$ if there is a first-order formula $\psi(A)$ in the expanded language such that $\<M,A>\satisfies\psi(A)$ and there are only finitely many $B\of M^n$ for which $\<M,B>\satisfies\psi(B)$. Allowing parameters from a fixed class $P\of M$ to appear in $\psi$ yields the notions of $P$-definability, implicit $P$-definability, and $P$-algebraicity in $M$. Simplifying the terminology, we say that $A$ is definable, implicitly definable, or algebraic {\df over} (rather than \emph{in}) $M$ if it is $M$-definable, implicitly $M$-definable, or $M$-algebraic in $M$, respectively. A natural generalization of these concepts arises by allowing second-order quantifiers to appear in $\psi$. Thus we may speak of a class $A$ as second-order definable, implicitly second-order definable, or second-order algebraic. Further generalizations are of course possible by allowing $\psi$ to use resources from other strong logics.

To begin our project, let us consider the class \HOA\ of hereditarily ordinal algebraic sets. In a strong second-order theory such as \KM, which proves the existence of satisfaction classes for first-order set-theoretic truth, we may stipulate explicitly that a set $a$ is {\df ordinal algebraic} if it is $\Ord$-algebraic in $V$, that is, if for some first-order formula $\varphi$ and ordinal parameters $\vec\alpha$ we have that $\varphi(a,\vec\alpha)$ holds and there are only finitely many $b$ for which $\varphi(b,\vec\alpha)$ holds. But because this definition invokes a truth predicate for first-order formulas, we cannot formalize it in the theory \GBC, which does not prove the existence of such a truth predicate, or in \ZFC, which by Tarski's non-definability theorem cannot have such a truth predicate. Of course, the same problem arises with the notion of ordinal definability, which one would like to characterize by saying that a set $a$ is ordinal definable if for some first-order $\varphi$ and ordinal parameters $\vec\alpha$ we have that $\varphi(b,\vec\alpha)$ holds just in case $b = a$. In that case, the familiar solution is to stipulate instead that a set $a$ is ordinal definable if it is $\theta$-definable in some $V_\theta$, that is, using ordinal parameters below $\theta$ and referring only to truth in $V_\theta$, a set. This condition can be formalized in \ZFC, and moreover (by the \Levy--Montague reflection theorem for first-order formulas) it is provably equivalent in \KM\ to the condition that $a$ is $\Ord$-definable in $V$.\footnote{See theorem 12.14 and equation 13.26 of \cite{Jech:SetTheory} on reflection and ordinal definability, respectively.} Similarly, in the present context we shall stipulate that a set $a$ is {\df ordinal algebraic} if it is $\theta$-algebraic in some $V_\theta$. Again, this can be formalized in \ZFC\ and is provably equivalent in \KM\ to the condition that $a$ is $\Ord$-algebraic in $V$. The point is that if $a$ is one of finitely many satisfying instances of $\varphi(\cdot, \vec\alpha)$ in $V$, then this fact will reflect to some $V_\theta$ with $\theta$ exceeding each of the parameters in $\vec\alpha$. Conversely, if $a$ is $\theta$-algebraic in $V_\theta$ via $\varphi(\cdot, \vec\alpha)$, then because $V_\theta$ is definable from parameter $\theta$, the object $a$ will be algebraic in $V$ via $\varphi(\cdot,\vec\alpha)^{V_\theta}$. Thus the collection \OA\ of ordinal algebraic sets is definable as a class in $V$. Likewise, then, for the collection \HOA\ of {\df hereditarily} ordinal algebraic sets, that is, ordinal algebraic sets whose transitive closures contain only ordinal algebraic sets.\footnote{
Nevertheless, there are some meta-mathematical subtleties to this approach. In the case of definability, suppose a model $M$ believes that $a$ is defined in $V_\theta^M$ by a formula $\varphi(\cdot,\vec\alpha)$. If $\varphi$ has standard length then, since $\theta\mapsto V_\theta^M$ is definable in $M$, we can see externally that $a$ is $\Ord^M$-definable in $M$ as the unique object satisfying the formula $\varphi(\cdot,\vec\alpha)^{V_\theta^M}$ with parameters $\vec\alpha$ and $\theta$. In fact, $a$ remains $\Ord^M$-definable in $M$ even when $\varphi$ is a nonstandard formula of $M$, for $a$ is the unique object thought by $M$ to satisfy, in $V_\theta^M$, the formula coded by the ordinal $\Godelnum{\varphi}$ with parameters coded by the ordinal $\Godelnum{\vec\alpha}$. Thus $\OD$, as defined in $M$ by our official definition, coincides with the class of objects in $M$ that really are $\Ord^M$-definable in $M$.

For algebraicity, however, there is a wrench in the works of the analogous absoluteness argument: although each object that is externally $\Ord^M$-algebraic in $M$ is also internally ordinal algebraic in $M$, it is conceivable that an $\omega$-nonstandard model $M$ may regard an object $a$ as ordinal algebraic even when it is not really $\Ord^M$-algebraic in $M$. This could happen if $n$ were a nonstandard integer and $M$ regarded $a$ as ordinal algebraic due to its membership in a set $\set{b \st V_\theta^M\satisfies\varphi(b,\vec\alpha)}$ of ``finite'' size $n$. In short, the more expansive concept of finiteness inside an $\omega$-nonstandard model $M$ may lead it to a more generous concept of algebraicity. Of course, since we do not yet know whether algebraicity differs from definability in any model of set theory, we cannot now confirm this possibility by presenting a particular $\omega$-nonstandard model $M$ for which $\OD^M$ contains sets that are not really $\Ord^M$-algebraic in $M$.}

\begin{theorem}\label{Theorem.HOA=HOD}
 The class of hereditarily ordinal algebraic sets is the same as the class of hereditarily ordinal definable sets: $$\HOA=\HOD.$$
\end{theorem}

\begin{proof}
Clearly $\HOD\of\HOA$. Conversely, we show by $\in$-induction that every element of $\HOA$ is in $\HOD$. Suppose that $a\in\HOA$ and assume inductively that every element of $a$ is in $\HOD$, so that $a\of\HOD$. Since $a$ is ordinal algebraic, there is an ordinal definable set $A = \set{a_0, \dots, a_n}$ containing it. We may assume that each $a_i $ is a subset of $\HOD$, by adding this condition to the definition of $A$. The definable well-ordering of $\HOD$ in $V$ gives rise to a definable bijection $h: \Ord \rightarrow \HOD$, where $h(\beta)$ is the $\beta^\th$ element of \HOD\ in that ordering.
Thus subsets of \HOD\ are definably identified via $h$ with sets of ordinals, and these in turn are definably linearly ordered, giving rise to a definable linear order $\prec$ on subsets of \HOD. Namely, $b \prec c$ if and only if the \HOD-least element of the symmetric difference of $b$ and $c$ is in $c$, that is, if $h(\min(h^{-1}b\mathrel{\triangle}h^{-1}c))\in c$, where $h^{-1}x=\set{\beta\st h(\beta)\in x}$. Thus $A$ is a finite ordinal definable set with a definable linear order $\prec$. So each $a_i$ is ordinal definable as the $k^\th$ element of $A$ with respect to $\prec$, for some $k$, using the same parameters as the definition of $A$ itself. In particular, $a$ is ordinal definable as desired.
\end{proof}

Since the foregoing proof uses the hereditary nature of \HOA, it does not seem to show $\OA=\OD$, and for now this question remains open. Perhaps it is consistent with \ZFC\ that some ordinal algebraic set is not ordinal definable. Or perhaps algebraicity simply coincides with definability in models of set theory. (Compare the previous footnote.) We are unsure.

A second application of definability in set theory concerns {\df pointwise definable} models, that is, structures in which each element is definable without parameters. For example, it is well known that the minimal transitive model of $\ZFC$ is pointwise definable, if it exists. Moreover, every countable model of $\ZFC$ has a pointwise definable extension by class forcing, and in fact every countable model of $\mathrm{GBC}$ has an extension by means of class forcing in which each set and each class is definable. Proofs of these facts and more appear in \cite{HamkinsLinetskyReitz:PointwiseDefinableModelsOfSetTheory}, which also gives references to earlier literature on the topic.

Here we define a structure to be pointwise algebraic, or simply {\df algebraic}, if each of its elements is algebraic in it. In some mathematical theories, this is not equivalent to pointwise definability. For instance, in the language of rings $\set{{+},{\cdot},0,1}$, every algebraic field extension of the rational field is algebraic in the sense just defined, since each of its elements is one of finitely many solutions to a particular polynomial equation over the integers, a property expressible in this language. But since such fields can have nontrivial automorphisms---a major focus of Galois theory---they can fail to be pointwise definable. Set theory is different:

\begin{theorem}\label{Theorem.AlgebraicIffPointwiseDefinable}
 Every algebraic model of \ZF\ is a pointwise definable model of $\ZFC+V=\HOD$.
\end{theorem}

\begin{proof}
If $M \vDash \ZF$ is algebraic, then of course each of its elements is ordinal algebraic in the sense of $M$. By theorem \ref{Theorem.HOA=HOD}, it follows that each element of $M$ is ordinal definable in the sense of $M$, and so $M\satisfies\ZFC+V=\HOD$. But each ordinal of $M$, being algebraic, belongs to a finite definable set of ordinals of $M$, a set that is definably linearly ordered by the membership relation of $M$. So each ordinal of $M$ is in fact definable in $M$. Since $M$ satisfies $V=\HOD$, this implies that every object in $M$ is definable, and so $M$ is pointwise definable. \end{proof}

\begin{corollary}\label{Corollary.WhenHasPointwiseAlgebraic}
An extension of $\ZF$ has an algebraic model if and only if it is consistent with $V = \HOD$.
\end{corollary}
\begin{proof}
In light of theorem \ref{Theorem.AlgebraicIffPointwiseDefinable}, we need only prove the right-to-left direction. So suppose $T$ extends $\ZF$ and is consistent with $V = \HOD$. Then some $M \vDash T$ satisfies $V = \HOD$. Since $M$ thinks the universe is definably well ordered, it has definable Skolem functions. Hence the $N\of M$ consisting of precisely the definable elements of $M$ is closed under these Skolem functions and is therefore an elementary substructure of $M$. So every element of $N$ is definable in $N$ by the same formula that defines it in $M$. Therefore $N \vDash T$ is pointwise definable and hence algebraic.
\end{proof}

\begin{corollary}
A complete extension of $\ZF$ has an algebraic model if and only if it has, up to isomorphism, a unique model in which each ordinal is definable.
\end{corollary}
\begin{proof}
By a result of Paris, a complete extension $T$ of $\ZF$ proves $V = \HOD$ if and only if it has, up to isomorphism, a unique model in which each ordinal is definable. (See theorem 3.6 of \cite{Enayat:ModelsOfSetTheoryWithDefinableOrdinals} for a proof.) In light of this, corollary \ref{Corollary.WhenHasPointwiseAlgebraic} yields the desired conclusion.
\end{proof}

Although theorem \ref{Theorem.AlgebraicIffPointwiseDefinable} shows that algebraicity and definability coincide in algebraic models of set theory, we do not know whether there can be a model of \ZF\ with an algebraic non-definable element.

Finally, we consider the algebraic analogue of \Godel's constructible universe. \Godel\ builds his universe $L$ by iterating the definable power set operation $\Pdef$, where for any structure $M$ the definable power set $\Pdef(M)$ consist of all classes $A\of M$ that are definable over $M$. Similarly, we define the {\df implicitly definable} power set of $M$ to be the collection $\Pimp(M)$ of all classes $A\of M$ that are implicitly definable over $M$, and we define the {\df algebraic} power set of $M$ to be the collection $\Palg(M)$ of all $A\of M$ that are algebraic over $M$.

\begin{observation}\label{Observation.AlgIffImplicit}
 The algebraic power set of a structure $M$ is identical to its implicitly definable power set: $$\Palg(M) = \Pimp(M).$$
\end{observation}

\begin{proof}
Obviously a class $A\of M$ is algebraic over $M$ if it is implicitly definable over $M$. For the converse, suppose $A\of M$ is algebraic over $M$ via $\varphi$. Note that each of the finitely many $B \neq A$ satisfying $\varphi$ in $M$ is distinguished from $A$ by some parameter $a \in M$ that is in $A$ but not $B$ or vice versa. Conjoining to $\varphi$ assertions about these parameters (either that $a \in A$ or $a \notin A$, as appropriate) produces a formula $\psi$ witnessing that $A$ is implicitly definable over $M$.
\end{proof}

Though algebraic classes are thus always implicitly definable, implicitly definable classes need not be (explicitly) definable. For example, if $M$ is an $\omega$-standard model of set theory, then its full satisfaction class $\text{Sat}^M=\set{\<\Godelnum{\varphi},\vec a>\st M\satisfies\varphi[\vec a]}$ is implicitly definable in $M$ as the unique class satisfying the familiar Tarskian recursive truth conditions. But Tarski's theorem on the non-definability of truth shows precisely that $\text{Sat}^M$ cannot be defined in $M$, even with parameters. It is interesting to note that this argument can fail when $M$ is $\omega$-nonstandard, for satisfaction classes need not be unique in such models, as shown in \cite{Krajewski1974:MutuallyInconsistentSatisfactionClasses} (see also further results in \cite{HamkinsYang:SatisfactionIsNotAbsolute}). 

Let us make a somewhat more attractive example, where the structure $M$ can be $\omega$-nonstandard and the relevant implicitly definable class $A$ must be {\df amenable} to $M$, meaning that $A\intersect a\in M$ for each $a\in M$. This is an improvement over the above example, for $\text{Sat}^M$ can fail to be amenable to $M$. (When $\text{Sat}^M$ is amenable to $M$, it follows that $\Th(M)$ is in $M$. But this fails, for example, in pointwise definable models of set theory, as mentioned in \cite{HamkinsLinetskyReitz:PointwiseDefinableModelsOfSetTheory}.) For the modified argument, let $N$ be any model of \ZFC, and let $\alpha_n$ be the least $\Sigma_n$-reflecting ordinal in $N$, that is, the least ordinal such that $(V_{\alpha_n})^N \prec_n N$. Such an $\alpha_n$ exists by the reflection theorem and is definable in $N$ using the definability of $\Sigma_n$ satisfaction. Let $M_n=(V_{\alpha_n})^N$ and let $M=\Union_n M_n$ be the union of the progressively elementary chain $M_0 \prec_0 M_1 \prec_1 \cdots$ of models. Since the union of a $\Sigma_n$-elementary chain is a $\Sigma_n$-elementary extension of each component of the chain, we have $M_n \prec_n M$ and so $M\prec N$. Note that $A=\set{\<\alpha_n,\Godelnum{\varphi},\vec a> \st \vec a\in M_n, \varphi\in\Sigma_n,n\in\omega, M_n\satisfies\varphi[\vec a]}$ is amenable to $M$. For if $a \in M$ then $a \in M_n$ for some $n$, whence $A\intersect a$ does not contain any triples $\<\alpha_m,\Godelnum{\varphi},\vec a>$ for $m\geq n$, and consequently $A\intersect a$ is constructible from $a$ and the $\Sigma_n$ satisfaction class of $M_n$ together with the finitely many parameters $\alpha_k$ for $k < n$, all of which are in $M$. Furthermore, $A$ is implicitly definable in $M$ using a version of the Tarskian satisfaction conditions. For any two truth predicates must agree that each $\alpha_n$ is least such that $(V_{\alpha_n})^M \prec_n M$, and the nonstandard formulas never get a chance to appear in $A$ as we have defined $M_n$ only for standard $n$ even when $N$ (and hence $M$ itself) is $\omega$-nonstandard. But $A$ cannot be definable in $M$, even with parameters, since from $A$ we can define a truth predicate for $M$.

Having appreciated these facts, we introduce the algebraic analogue of $L$, the {\df implicitly constructible universe}, hereby dubbed \Imp\ and built as follows:
\begin{align*}
 \Imp_0 &= \emptyset \quad &\Imp_{\alpha + 1} &= \Pimp(\Imp_\alpha)\\
 \Imp_\lambda &= \bigcup_{\alpha < \lambda} \Imp_\alpha, \mbox{ for limit $\lambda$} \quad &\Imp &= \bigcup_\alpha \Imp_\alpha.
\end{align*}

\begin{theorem}\label{Theorem.ImpSatisfiesZF}
$\Imp$ is an inner model of $\ZF$ with $L\of\Imp\of\HOD$.
\end{theorem}

\begin{proof}
Clearly \Imp\ is a transitive class containing all ordinals and closed under the \Godel\ operations. \Imp\ is almost universal as well, for each of its subsets is included in some $\Imp_\alpha$, each of which belongs to $\Imp$. Any such class is an inner model of \ZF.\footnote{See theorem 13.9 of \cite{Jech:SetTheory} for a definition of \emph{almost universality} and a proof that any almost universal transitive class containing all ordinals and closed under the \Godel\ operations is a model of \ZF.} Since $L$ is the least such inner model, we therefore have $L\of\Imp$. To see that $\Imp\of\HOD$, recall from \cite{MyhillScott:OrdinalDefinability} that \HOD\ is identical to the class obtained by transfinite iteration of the second-order definable power set operation. (In other words, \HOD\ is just the second-order constructible universe.) Since the second-order definable power set of a structure $M$ includes the implicitly definable power set $\Pimp(M)$, the inclusion $\Imp\of\HOD$ is immediate. One can also obtain this inclusion by transfinite induction: if $\Imp_\alpha\of\HOD$ and $A\in\Imp_{\alpha+1}$, then $A$ is implicitly definable over $\Imp_\alpha$ by a specific formula with parameters from $\Imp_\alpha$, and since these parameters and $\Imp_\alpha$ are ordinal definable, it follows that $A$ is ordinal definable as the unique subset of $\Imp_\alpha$ that is implicitly defined by that formula from those parameters. That is, although $A$ is only implicitly definable as a subset of $\Imp_\alpha$, this fact serves as a first-order definition of $A$ as an element of $V$.
\end{proof}

We are unsure whether $\Imp$ must satisfy the axiom of choice. This is related to the subtle issue of whether $\Imp^\Imp=\Imp$, that is, whether $\Imp$ can see that it is $\Imp$. If so then \Imp\ would know that $\Imp^\Imp \of \HOD$ and so would satisfy the statement $V = \HOD$, which implies the axiom of choice. In essence, if $\Imp$ can see that it is $\Imp$, then it can define a well-ordering of the universe: one set precedes another when it appears earlier in the $\Imp$ hierarchy or at the same time but with a smaller formula, or with the same formula but with earlier parameters.

Unfortunately, we do not know whether $\Imp^\Imp=\Imp$ must hold. To highlight a difficulty, note that one might aim to prove $\Imp^\Imp=\Imp$ by showing inductively that it is true in a level-by-level manner, that is, by proving $\Imp_\alpha^\Imp=\Imp_\alpha$ for each $\alpha$. Of course, if this identity holds at $\alpha$ then $\Imp_{\alpha+1}\of\Imp_{\alpha+1}^\Imp$, because any $A\in\Imp_{\alpha+1}$ will be implicitly definable over $\Imp_\alpha^\Imp$ and contained in $\Imp$, and so it will be in $\Imp_{\alpha+1}^\Imp$. But the problem for the converse is that perhaps some $B \of \Imp_\alpha^\Imp$, belonging not to $\Imp_\alpha$ but still belonging to some later stage $\Imp_\beta$, will be implicitly definable over $\Imp_\alpha$ in the sense of \Imp\ but not in the sense of $V$. This will happen, for example, if $V \neq \Imp$ and every formula witnessing in \Imp\ that $B$ is implicitly definable over $\Imp_\alpha$ is satisfied in $V$ also by some set other than $B$. So we seem not to be able to argue that $\Imp_{\alpha+1}^\Imp\of\Imp_{\alpha+1}$, and the inductive method of showing $\Imp^\Imp=\Imp$ therefore appears to break down.

Putting that aside for the moment, let us examine the relationship between \Imp\ and $L$. Observe first that for any countable structure $M$, the statement that $A\of M$ is implicitly definable over $M$ is a $\Pi^1_1$ assertion in the codes for $A$ and $M$. Shoenfield absoluteness therefore ensures that $\Imp_\alpha$ is absolute from $V$ to $L$ for $\alpha\lt\omega_1^L$, and so $\Imp_{\omega_1^L}=(\Imp_{\omega_1})^L=L_{\omega_1^L}$. Meanwhile, the $\Imp_\alpha$ hierarchy grows faster than the $L_\alpha$ hierarchy. For as noted above, the satisfaction class $\set{\<\Godelnum{\varphi},\vec a>\st \Imp_\alpha\vDash\varphi[\vec a]}$ is implicitly, but not explicitly, definable over $\Imp_\alpha$. Furthermore, the satisfaction class for hyperarithmetic truth is in $\Pimp(\Imp_\omega) = \Imp_{\omega+1}$ but does not appear in $L$ until stage $\omega_1^{\rm CK}$. The satisfaction relation for $L_{\kappa,\lambda}$ logic over $\Imp_\alpha$, using formulas in $\Imp_\alpha$, is likewise implicitly definable in $\Imp_\alpha$. One naturally wonders whether the $L_\alpha$ hierarchy ever catches up to the $\Imp_\alpha$ hierarchy, in the sense that each $\Imp_\alpha$ is contained in some $L_\beta$. Not necessarily:

\begin{theorem}\label{Theorem.LalgNeqL}
 It is relatively consistent with \ZFC\ that $\Imp\neq L$.
\end{theorem}

\begin{proof} Let $T$ be a Souslin tree in $L$ with the unique branch property, a strong notion of rigidity described in \cite{FuchsHamkins2009:DegreesOfRigidity}. Forcing over $L$ with $T$ yields a model $L[b]$ containing exactly one cofinal branch $b$ through $T$. Since $T\in L$, there is an $\alpha$ with $T \in \Imp_\alpha^{L[b]}$. And in $\Imp_\alpha^{L[b]}$ the formula ``$X$ is a cofinal branch through $T$'' is satisfied uniquely by $b$. So $b \in \Imp_{\alpha+1}^{L[b]}$. But $b\notin L$. So $L[b]$ thinks $\Imp\neq L$.
\end{proof}

\begin{corollary}
 $\Imp$ is not absolute to forcing extensions.
\end{corollary}

A careful inspection of the proof of theorem \ref{Theorem.LalgNeqL} shows that a copy of $T$ becomes definable in $\Imp_{\omega_1}^{L[b]}$, and so in fact $b\in\Imp_{\omega_1+1}^{L[b]}$. Moreover, the branch $b$ is implicitly definable over $L$ inside any universe in which $b$ is isolated in $T$. In general, we would like to know what else is contained in $\Pimp(L)$. For example, is $0^\sharp$, considered as a set of natural numbers, implicitly definable over $L$? Can any non-constructible real be implicitly definable over $L$?

Refining the notion of implicit constructibility captured by \Imp, we define the class \gImp, the {\df generic} implicitly constructible universe, to consist of those sets $a$ that belong to $\Imp^{V[G]}$ for every set-forcing extension $V[G]$ of the universe. This class is first-order definable in $V$. One can compare \gImp\ with the generic \HOD\ class $\gHOD$, an inner model of \ZFC\ defined in \cite{FuchsHamkinsReitz:Set-theoreticGeology}. Since the proof of theorem \ref{Theorem.ImpSatisfiesZF} works in each forcing extension, we note that $\gImp\of\gHOD$.

\begin{theorem}
In any set-forcing extension $L[G]$ of $L$, there is a further extension $L[G][H]$ with $\gImp^{L[G][H]}=\Imp^{L[G][H]}=L$.
\end{theorem}

\begin{proof}
Let $L[G][H]$ be the forcing extension obtained by absorbing the $G$ forcing into a large collapse $\Coll(\omega,\theta)$ forcing that is almost homogeneous. Since $L$ is the \HOD\ of $L[G][H]$, it follows that $\gImp^{L[G][H]}=\Imp^{L[G][H]}=L$ as well.
\end{proof}

Open questions about \Imp\ abound. Can $\Imp^\Imp$ differ from $\Imp$? Does \Imp\ satisfy the axiom of choice? Can \Imp\ have measurable cardinals? Must $0^\sharp$ be in \Imp\ when it exists?\footnote{An affirmative answer arose in conversation with Menachem Magidor and Gunter Fuchs, and we hope that \Imp\ will subsume further large cardinal features. We anticipate a future article on the implicitly constructible universe.} Which large cardinals are absolute to \Imp? Does \Imp\ have fine structure? Should we hope for any condensation-like principle? Can \CH\ or \GCH\ fail in \Imp? Can reals be added at uncountable construction stages of \Imp? Can we separate \Imp\ from \HOD? How much can we control \Imp\ by forcing? Can we put arbitrary sets into the \Imp\ of a suitable forcing extension? What can be said about the universe $\Imp(\mathbb{R})$ of sets implicitly constructible relative to $\mathbb{R}$ and, more generally, about $\Imp(X)$ for other sets $X$? Here we hope at least to have aroused interest in these questions.

\bibliographystyle{alpha}

\end{document}